\documentclass[12pt]{amsart}
\usepackage{amssymb}
\usepackage{amsmath}
\usepackage{amssymb}
\usepackage{amsfonts}
\usepackage{latexsym}
\usepackage{bm}
\usepackage{xcolor}
\usepackage{mathrsfs}
\usepackage{graphicx}
\usepackage{graphics}
\usepackage{enumerate}
\usepackage{mathtools}
\usepackage[normalem]{ulem}

\newcommand{\vp}{\varphi}

\newcommand{\mb}{\mathbb}
\newcommand{\A}{{\mathfrak A}}

\newcommand{\M}{{\mathfrak M}}
\newcommand{\Ao}{{\mathfrak A}_0}

\newcommand{\mc}{\mathcal}

\newcommand{\D}{{\mc D}}

\newcommand{\id}{{I}}

\newcommand{\Hil}{\mc H}

\newcommand{\AC}{\mathscr{A}}
\newcommand{\ACo}{\mathscr{A}_0}
\newcommand{\PA}{{\mathcal P}_{\Ao}(\A)}
\newcommand{\SSA}{{\mathcal S}_{\Ao}(\A)}
\newcommand{\mult}{\,{\scriptstyle \square}\,}
\newtheorem{theorem}{Theorem}[section]

\newtheorem{corollary}[theorem]{Corollary}
\newtheorem{lemma}[theorem]{Lemma}
\newtheorem{proposition}[theorem]{Proposition}

\theoremstyle{definition}
\newtheorem{defn}[theorem]{Definition}
\newtheorem{remark}[theorem]{Remark}
\newtheorem{example}[theorem]{Example}

\newtheorem*{xrem}{Remark}

\numberwithin{equation}{section}

\begin{document}
%\renewcommand{\baselinestretch}{2}

%\begin{frontmatter}



\title{Eigenstates of CQ*-algebras}%



\author[F. Bagarello]{Fabio Bagarello}

\address{Dipartimento di Ingegneria, Universit\`{a} di Palermo, I - 90128 Palermo, and I.N.F.N., Sezione di Catania, Italy. %(e-mail:  fabio.bagarello@unipa.it).
}
\email{fabio.bagarello@unipa.it}

\author[H. Inoue]{Hiroshi Inoue}

\address{Department of Economics, Kyushu Sangyo University, 2-3-1 Matsukadai Higashi-ku, Fukuoka, Japan. %(e-mail: h-inoue@ip.kyusan-u.ac.jp).
}
\email{h-inoue@ip.kyusan-u.ac.jp}
\author[C. Trapani]{Camillo Trapani}
\address{Dipartimento di Matematica e Informatica, Universit\`a degli Studi  di Palermo, Via Archirafi n. 34,  I-90123 Palermo, Italy. %(e-mail:camillo.trapani@unipa.it).
}
\email{camillo.trapani@unipa.it}

\author[S. Triolo]{Salvatore Triolo}

\address{Dipartimento di Ingegneria, Universit\`{a} di Palermo, I - 90128 Palermo, Italy. %(e-mail: salvatore.triolo@unipa.it)
}
\email{salvatore.triolo@unipa.it}
\date{}
\maketitle
\begin{abstract}
Motivated by some recent results, we consider the notion of eigenstate (and eigenvalue) for an element $X$ of a CQ*-algebras and the consequences on algebraic quantum dynamics and on its related derivations are investigated. 
\end{abstract}


%\end{frontmatter}

% \linenumbers

\section{Introduction} 
The notion of eigenvector, which is very familiar when one deals with linear  operators, has been considered in the abstract setting of C*-algebras \cite{3}, using positive linear functionals. Extensions of this approach have been proposed in \cite{bagintri} and \cite{8} to more general contexts, with the aim of using them in the mathematical description of quantum systems where unbounded operators appear in a natural fashion. In particular in \cite{bagintri} the attention has been focused to the case of quasi*-algebras (see \cite{10} for a synthesis on this subject) where {\em eigenstates} have been described through certain {\em invariant} positive sesquilinear forms, shortly, ips-forms. Their main feature consists in the fact that they allow a GNS construction similar to that induced by positive linear functional (or states) on *-algebras and this is clearly an essential tool when one wants to pass from abstract *-algebras or quasi *-algebras to concrete realizations with operators. An interesting application discussed in \cite{bagintri} is related to ladder elements which reproduce, at an algebraic level, the (pseudo-)bosonic commutation relations, \cite{bagspringer}.

In this paper we consider the case of eigenstates of a  CQ*-algebra.  This structure is obtained, roughly speaking, by taking the completion $\A$ of a C*-algebra $\Ao$ under a norm $\|\cdot\|$, weaker than the original norm $\|\cdot\|_0$ of $\Ao$ and enjoying some additional properties, { coupling the two norms}.

In Section 3, we shortly discuss positive linear functionals on a CQ*-algebra
obtained by extending to $\A$ positive linear functionals on $\Ao$ 
which are continuous with respect to the norm $\|\cdot\|$ of $\A$. It is shown how the GNS construction can be adapted to this situation to get a *-representation of $\A$. This is possible but one has to pay a (little) price: the use of the notion of {\em unbounded vector} (due to M.Tomita \cite{7}) involves representations that live beyond the Hilbert space.
The notion of eigenvalue and eigenstate are then introduced in Section 4 using the positive linear functionals introduced in Section 3 (Section 4.1) or {invariant positive sesquilinear} (ips) forms (Section 4.2) as in \cite{bagintri}. Here we consider  the case of a *-semisimple CQ*-algebra (which by definition possesses a sufficient number of bounded ips-forms). In this case, the treatment remains within Hilbert spaces.    

Section 5 is devoted to the study of the role played by eigenvectors and eigenstates for the dynamics both at integral level (*-automorphisms) and at the infinitesimal one (*-derivations); As shown in the paper, several classical properties, well known for C*-algebras, generalize to our environment, under appropriate (but light) assumptions.

Finally, in Section 6, starting from a CQ*-algebra $(\A[\|\cdot\|], \Ao$)  we propose the construction of a locally convex *-algebra $\A_1$, with $\Ao\subset \A_1\subset \A$ which has the property that every $\|\cdot\|$-continuous positive linear functional on $\Ao$ extends to an {\em admissible} positive linear functional on $\A_1$. This is quite a well behaved situation, since admissible positive linear functional give rise, via GNS construction, to bounded operators. Section \ref{sect7} contains our conclusions.

\section{Preliminaries}

%%%%%%%%%%%%%%%%%%%%%%%%%%%%%%%%%%%%%%%%%%%%%%%%%%%%%%%%%%%%%%%%%%%%%%%%%%
{A {\em quasi *-algebra} $(\AC, \ACo)$ is a pair consisting of a vector space $\AC$ and a *-algebra $\ACo$ contained in $\AC$ as a subspace and such that 
\begin{itemize}
	\item $\AC$ carries an involution $a\mapsto a^*$ extending the involution of $\ACo$;
	\item $\AC$ is  a bimodule over $\AC_0$ and the module multiplications extend the multiplication of $\ACo$. In particular, the following associative laws hold:
	\begin{equation}\notag \label{eq_associativity}
		(ca)d = c(ad); \ \ a(cd)= (ac)d, \quad \forall \ a \in \AC, \  c,d \in \ACo;
	\end{equation}
	\item $(ac)^*=c^*a^*$, for every $a \in \AC$ and $c \in \ACo$.
\end{itemize}

The
\emph{identity} or {\it unit element} of $(\AC, \ACo)$, if any, is a necessarily unique element $\id\in \ACo$, such that
$a\id=a=\id a$, for all $a \in \AC$.

We will always suppose that
\begin{align*}
	&ac=0, \; \forall c\in \ACo \Rightarrow a=0 \\
	&ac=0, \; \forall a\in \AC \Rightarrow c=0. 
\end{align*}
Clearly, both these conditions are automatically satisfied if $(\AC, \ACo)$ has an identity $\id$.\\}

%%%%%%%%%%%%%%%%%%%%%%%%%%%%%%%%%%%%%%%%%%%%%%%%%%%%%%%%%%%%%%%%%%%%%%%%%%%
%Here we review the notion of quasi $\ast$-algebras.
%Let $\mathscr{A}$ be a quasi $\ast$-algebra over $\mathscr{A}_0$, that is,
%\begin{itemize}
%\item $\mathscr{A}$ is a vector space which contains a $\ast$-algebra $\mathscr{A}_0$
%\item multiplications $ax$ and $xa$ of $a\in\mathscr{A}_0 $ and $x\in\mathscr{A}$ having the following properties are defined:
%\begin{itemize}
%\item distributive law:
%\begin{eqnarray}
%(x+y)a
%&=& xa+ya, \;\;\;
%a(x+y)
%= ax+ay\nonumber \\
%x(a+b)
%&=& xa+xb,\;\;\;
%(a+b)x
%= ax+bx\nonumber
%\end{eqnarray}
%for all $a,b\in \mathscr{A}_0$ and $x,y\in\mathscr{A}$.
%\item scalar multiplication:
%\begin{eqnarray}
%\alpha(xa)
%&=& (\alpha x)a=x(\alpha a) \nonumber \\
%\alpha(ax)
%&=& (\alpha a)x=a(\alpha x) \nonumber
%\end{eqnarray}
%for all $a\in\mathscr{A}_0$, $x\in\mathscr{A}$ and $\alpha\in\mathbb{C}$.
%\item associative law:
%\begin{eqnarray}
%a(bx)
%= (ab)x,\;
%a(xb)
%= (ax)b,\;
%x(ab)
%= (xa)b\nonumber
%\end{eqnarray}
%for all $a,b\in\mathscr{A}_0$ and $x\in\mathscr{A}$.
%\end{itemize}
%\item involution $\ast$ on $\A$ which is an extention of involution $\ast$ on $\Ao$
%\begin{eqnarray*}
%(x+y)^\ast&=&x^\ast+y^\ast, \;
%(\alpha x)^\ast=\bar{\alpha}x^\ast\\
%(x^\ast)^\ast&=&x,\;
%(ax)^\ast=x^\ast a^\ast,\; (xa)^\ast=a^\ast x^\ast
%\end{eqnarray*}
%for all $a\in\mathscr{A}_0$, $x,y\in\mathscr{A}$ and $\alpha\in\mathbb{C}$.
%\end{itemize}
%%%%%%%%%%%

Let $\Ao$ be an unital C*-algebra with C*-norm $\| \cdot \|_0$. Assume that another  norm $\| \cdot\|$ is defined on $\Ao$, satisfying the following properties:
\begin{enumerate}[(i)]
	\item $\|A\|\leqq \|A\|_0$, for all $A\in \Ao$
	\item $\|AB\| \leqq \|A\|\|B\|_0$, for all $A,B\in \Ao$
	\item $\|A^\ast\|=\|A\|$, for all $A\in \Ao$.
\end{enumerate}
By (ii) and (iii) we have
\begin{enumerate}[(iii)$^\prime$]
	\item $\|AB\|\leqq \|A\|_0\|B\|$, for all $A,B\in \Ao$.
\end{enumerate}
{We denote by $\A$ the completion of the normed space $(\Ao,\| \cdot\|)$. 
For any $X\in \A$ we put 
\begin{eqnarray}
	\| X\|^\sim =\lim_{n\rightarrow \infty}\| A_n\|,\nonumber
\end{eqnarray}
{where $\{ A_n\}$ is a sequence in $\Ao$ with $\|\cdot\|$-$\lim_{n\rightarrow\infty}A_n=X$. As usual, the extension $\|\cdot\|^\sim$ on $\A$ of the norm  $\|\cdot\|$ of $\Ao$, will simply be denoted by the same symbol $\|\cdot\|$.}
As shown in \cite[Proposition 5.1,3]{10} $(\A[\|\cdot\|], \Ao )$ is a (proper) {\em CQ*-algebra}, shortly, CQ*-algebra. We often say also that $\A$ is  a CQ*-{\em algebra over $\Ao$}.
The pair $(\A, \Ao)$ is a  quasi $\ast$-algebra
with the following multiplications and involution $\ast$:\\
For $X\in \A$ and $A\in \Ao$
\begin{itemize}
\item $XA := \|\cdot \|$-$\lim_{n\rightarrow\infty}A_nA$, $AX:= \|\cdot\|$-$\lim_{n\rightarrow\infty}AA_n$
\item $X^\ast:= \|\cdot\|$-$\lim_{n\rightarrow\infty}A_n^\ast$,
\end{itemize}
where $\{ A_n\}$ is a sequence in $\Ao$ with $\|\cdot\|$-$\lim_{n\rightarrow\infty}A_n=X$, and it satisfies
$$
\|XA\|\leqq\|X\|\|A\|_0, \;\;\; \|AX\|\leq\|A\|_0\|X\|, \quad \|X^\ast\|=\|X\|.
$$}
%and
%\begin{eqnarray}
%.\nonumber
%\end{eqnarray}
%%\end{proposition}
%%Such a $\A$ is called a CQ*-{\em algebra over $\Ao$}.
\begin{example}\label{esempio ellepi}
  The space $L^p([0,1])$, with $1\leq p< +\infty$ is a Banach $L^\infty([0,1])$-bimodule. The couple
$(L^p([0,1]), L^\infty([0,1])$ may be regarded as an {abelian}
CQ*-algebra.

\end{example}
\begin{example}\label{esempio semi}
Let $\M$ be a von Neumann algebra  on a Hilbert space $\Hil$ and $\vp$ a normal faithful
semifinite trace defined on $\M_+$. For each $p\geq 1$, let
$${\mc J_p}=\{ X \in \M: \vp(|X|^p)<\infty \}.$$ Then ${\mc J_p}$ is a
*-ideal of $\M$. Following \cite{nelson}, we denote with
$L^p(\vp)$ the Banach space completion of ${\mc J_p}$ with respect
to the norm
$$\|X\|_{ {p, \vp}} := \vp(|X|^p)^{1/p}, \quad X \in {\mc J_p}.$$
One usually defines $L^\infty(\vp): = \M$. Thus, if $\vp$ is a
finite trace, then $L^\infty(\vp) \subset L^p(\vp)$ for every
$p\geq 1$. As shown in \cite{nelson}, if $X \in L^p(\vp)$, then
$X$ is a measurable operator (see \cite{nelson}). In this case $(L^\infty(\vp), L^p(\vp))$ may be regarded as a
CQ*-algebra.

\end{example}
\section{Positive linear functionals of the CQ*-algebra $\A$}

Let $\omega$ be a $\|\cdot\|$-continuous positive linear functional on $\Ao$, that is,
\begin{equation}
{\rm There \; exists}\; \gamma>0; \;\; |\omega(A)|\leqq \gamma\|A\| \;\; {\rm for \; all}\; A\in \Ao. \label{3.1}
\end{equation}
We put
\begin{eqnarray}
\overline{\omega}(X)=\lim_{n\rightarrow\infty}\omega(A_n), \;\; X\in \A . \nonumber
\end{eqnarray}
where $\{ A_n\}$ is a sequence in $\Ao$ such that $\|\cdot\|$-$\lim_{n\rightarrow\infty}A_n=X$.
Then $\overline{\omega}$ is well-defined, that is, $\lim_{n\rightarrow\infty}\omega(A_n)$ exists in $\mathbb{C}$ and $\overline{\omega}(X)$ does not depend on the choice of the sequence $\{ A_n\}$ in $\Ao$, and it is a $\|\cdot\|$-continuous hermitian linear functional on $\A$ which is an extension of $\omega$.
%But, we can not consider positivity of $\overline{\omega}$ because $X^\ast X$ is not defined.

We put $\A^+:=\overline{\Ao^+}$, {the $\|\cdot\|$-closure of the set of positive elements of $\Ao$.}

 \begin{defn}
A linear functional $\omega$ which is defined on $\A$ will be called positive if $\omega(A) \geq 0$, for every $A\in \A^+$.
\end{defn}
This implies that $\omega$ is continuous on positive elements \cite[Lemma 3.1.48]{10}.

From this definition, $\overline{\omega}$ is a positive linear functional on $(\A, \Ao)$.

We shall consider a GNS-construction for $\overline{\omega}$.
Let $(\pi_\omega,\lambda_\omega,\Hil_\omega)$ be the GNS-construction for the positive linear functional $\omega$ on the C*-algebra $\Ao$, that is, $\pi_\omega$ is a $\ast$-representation of $\Ao$ into the C*-algebra $B(\Hil_\omega)$ of all bounded linear operators on a Hilbert space $\Hil_\omega$ and $\lambda_\omega$ is a vector representation of $\Ao$ in $\Hil_\omega$, that is, it is a linear mapping of $\Ao$ onto the dense subspace $\lambda_\omega(\Ao)$ in $\Hil_\omega$ satisfying $\lambda_\omega(AB)=\pi_\omega(A)\lambda_\omega(B)$ for all $A,B\in \Ao$.
Here we denote by $(\cdot |\cdot)$ the inner product of a Hilbert space $\Hil_{\omega}$.
For any $A,B\in \Ao$ we have
\begin{equation}
\|\pi_\omega(A)\lambda_\omega(B)\|^2=\omega(B^\ast A^\ast AB) \leqq \gamma \| B\|_0^2 \|A^\ast A\|. \label{3.2}
\end{equation}
{Take an arbitrary $X\in \A$ and let
 $\{ A_n\}$ be a sequence in $\Ao$, $\|\cdot\|$-converging to $X$.}
By \eqref{3.2} we have
\begin{eqnarray}
\| \pi_\omega(A_m)\lambda_\omega(B)-\pi_\omega(A_n)\lambda_\omega(B)\|^2
\leqq \gamma \|B\|_0^2 \|(A_m-A_n)^\ast (A_m-A_n)\|, \nonumber
\end{eqnarray}
but because the multiplication $AB$ is not $\|\cdot\|$-continuous, $$\lim_{m,n\rightarrow\infty}\|(A_m-A_n)^\ast(A_m-A_n)\|\neq 0$$ in general.
Hence $\lim_{n\rightarrow\infty}\pi_\omega(A_n)\lambda_\omega(B)$ may fail to exist in $\Hil_\omega$.
For this reason, we need to generalize the usual operator representations to form representations.
For that, we define the notions of unbounded vectors in a Hilbert space \cite{7}.
Let $\Hil$ be a Hilbert space.
Following M. Tomita we say that a conjugate linear functional $v$,  defined in a subspace $\D$ of $\Hil$, is an {\em unbounded vector} in $\Hil$ with domain $\D$.  The value of $v$ at $\xi$ in $\D$ is denoted by $<v,\xi>$.
 We denote by $v^\ast$ the complex conjugation, that is, $<v^\ast,\xi>=\overline{<v,\xi>}$, $\xi\in\D$.
Then $v^\ast$ is a linear functional on $\D$. We denote by $v(\D)$ the set of all unbounded vectors in $\Hil$ with domain $\D$. Then $v(\D)$ is a vector space under the operations:
\begin{eqnarray}
<v_1+v_2,\cdot>
&=& <v_1,\cdot>+<v_2,\cdot>, \nonumber \\
<\alpha v,\cdot>
&=& \alpha <v,\cdot>\nonumber
\end{eqnarray}
for $v,v_1,v_2\in v(\D)$ and $\alpha\in\mathbb{C}$.
An unbounded vector $v$ in $v(\D)$ is called {\em bounded} if $\D$ is dense in $\Hil$ and it can be extended to a continuous conjugate linear functional on $\Hil$.
Then the extension of $v$ is identified with the element of $\Hil$ and it is denoted by $[v]$.
Here let us denote by $\mathcal{L}^\dagger(\D,v(\D))$ the set of all linear mappings from $\D$ to $v(\D)$.
Then $\mathcal{L}^\dagger(\D,v(\D))$ is a quasi $\ast$-algebra over $\mathcal{L}^\dagger(\D)$ equipped with the following operations and involution $X\rightarrow X^\dagger$: for $X_1,X_2\in v(\D)$, $A\in\mathcal{L}^\dagger(\D)$ and $\xi,\eta\in\D$
\begin{eqnarray}
(X_1+X_2)\xi
&=& X_1\xi+X_2\xi,  \nonumber \\
(\alpha X)\xi
&=& \alpha(X\xi),  \nonumber \\
<AX\xi,\eta>
&=& <X\xi,A^\dagger \eta>, \nonumber \\
<XA\xi,\eta>
&=& <X(A\xi),\eta>, \nonumber
\end{eqnarray}
and
\begin{eqnarray}
<X^\dagger \xi,\eta>=\overline{<X\eta,\xi>}.\nonumber
\end{eqnarray}
\begin{defn} \label{def31}
Let $\mathscr{A}$ be a quasi $\ast$-algebra over $\mathscr{A}_0$. A linear mapping $\pi$ of $\mathscr{A}$ into $\mathcal{L}^\dagger(\D,\mathcal{V}(\D))$ is said to be a $\ast$-representation of $\mathscr{A}$ into $\mathcal{L}^\dagger(\D,\mathcal{V}(\D))$ if $\pi(ax)=\pi(a)\pi(x)$, $\pi(xa)=\pi(x)\pi(a)$ and $\pi(x^\ast)=\pi(x)^\dagger$ for all $x\in \mathscr{A}$ and $a\in\mathscr{A}_0$. Here we denote $\D$ and $\Hil$ by $\D(\pi)$ and $\Hil_\pi$, respectively.
\end{defn}
Let $\pi$ be a $\ast$-representation of $\mathscr{A}$ into $\mathcal{L}^\dagger(\D(\pi),\Hil_\pi)$.
Then $\pi(\mathscr{A})$ is a quasi $\ast$-algebra over $\pi(\mathscr{A}_0)$.
For a GNS-construction of $\A$ for $\overline{\omega}$ we have the following
\begin{proposition} \label{prop32}
Let $\omega$ be a $\|\cdot\|$-continuous positive linear functional on $\Ao$. We can define a triple $(\pi_{\overline{\omega}},\lambda_{\overline{\omega}},\Hil_{\overline{\omega}})$ satisfying
\begin{itemize}
\item $\pi_{\overline{\omega}}$ is a $\ast$-representation of $\A$ into $\mathcal{L}^\dagger(\lambda_\omega(\Ao),v(\lambda_\omega(\Ao)))$.
\item $\lambda_{\overline{\omega}}$ is a linear mapping from $\A$ to $v(\lambda_\omega(\Ao))$ satisfying
\begin{eqnarray}
\lambda_{\overline{\omega}}(XB)=\pi_{\overline{\omega}}(X)\lambda_\omega(B)\nonumber
\end{eqnarray}
for all $X\in \A$ and $B\in \Ao$.
\end{itemize}
Here $\Hil_{\overline{\omega}}=\Hil_\omega$, $\D(\pi_{\overline{\omega}})=\lambda_\omega(\Ao)$, and $\pi_{\overline{\omega}}$ and $\lambda_{\overline{\omega}}$ are extensions of $\pi_\omega$ and $\lambda_\omega$, respectively.
\end{proposition}

\begin{proof}
For any $A,B,C \in \Ao$ we have
\begin{align}
|(\pi_\omega(A)\lambda_\omega(B)|\lambda_\omega(C))|
&= |\omega(C^\ast AB)|\nonumber\\
&\leqq \gamma \|C^\ast AB\| \nonumber\\
&\leqq \gamma \|B\|_0\|C\|_0\|A\|. \label{3.3}
\end{align}
Take an arbitrary $X\in \A$. Let $\{ A_n\}$ be a sequence in $\Ao$ which $\|\cdot\|$-$\lim_{n\rightarrow\infty}A_n=X$.
By \eqref{3.3} we have
\begin{eqnarray}
\lim_{m,n\rightarrow\infty}|(\pi_\omega(A_m)\lambda_\omega(B)|\lambda_\omega(C))-(\pi_\omega(A_n)\lambda_\omega(B)|\lambda_\omega(C)) |\nonumber \\
\leqq \gamma \|B\|_0\|C\|_0 \lim_{m,n\rightarrow\infty}\|A_m-A_n\|=0\nonumber
\end{eqnarray}
for all $B,C\in \Ao.$ We can define a linear mapping $\pi_{\overline{\omega}}(X)$ from $\lambda_\omega(\Ao)$ to $\mathcal{V}(\lambda_\omega(\Ao))$ by
\begin{eqnarray}
<\pi_{\overline{\omega}}(X)\lambda_\omega(B),\lambda_\omega (C))>
:= \lim_{n\rightarrow\infty}(\pi_\omega(A_n)\lambda_\omega(B)|\lambda_\omega(C))\nonumber
\end{eqnarray}
for $B,C\in \Ao$. Then it is easily shown that $\pi_{\overline{\omega}}$ is a $\ast$-representation of $\A$ into $\mathcal{L}^\dagger(\lambda_\omega(\Ao),\mathcal{V}(\lambda_\omega(\Ao)))$ which is an extension of $\pi_\omega$ and $\lambda_{\overline{\omega}}$ is a vector representation of $\A$ into $\mathcal{V}(\lambda_\omega(\Ao))$ satisfying
\begin{eqnarray}
\lambda_{\overline{\omega}}(XB)
=\pi_{\overline{\omega}}(X) \lambda_\omega(B)\nonumber
\end{eqnarray}
for all $B\in \Ao$, which is an extension of $\lambda_\omega$. This completes the proof.
\end{proof}

The triple $(\pi_{\overline{\omega}},\lambda_{\overline{\omega}}, \Hil_{\overline{\omega}})$ in Proposition \ref{prop32} is called the GNS-construction of $\A$ for $\overline{\omega}$.

\section{Eigenstate}

\subsection{Eigenstates and spectrums}
Let $\omega$ be a $\|\cdot\|$-continuous positive linear functional on $\Ao$. If $\omega(I)=1$, then $\omega$ is called a  {\em state} of $\Ao$. If $\omega$ is a state of $\Ao$, then $\overline{\omega}$ is state of $\A$. We denote by $E(\Ao)$ (resp. $E(\A)$) the set of all $\|\cdot\|$-continuous states of $\Ao$ (resp. $\A$). Then
\begin{eqnarray}
\omega\in E(\Ao) \mapsto \overline{\omega}\in E(\A) \nonumber
\end{eqnarray}
is a bijection. In analogy with \cite{3,bagintri,8} we use the following definition of eigenstate and eigenvalue of an element $X$ of the CQ*-algebra: 
\begin{defn} \label{def41}
Let $X\in \A$. $\overline{\omega}$ is said to be an eigenstate of $X$ with eigenvalue $\alpha$ if
$\overline{\omega}(AX)=\alpha \overline{\omega}(A)$ for all $A\in \Ao$. The set of all eigenvalues of $X$ is denoted by $Eig(X)$.
\end{defn}
\begin{lemma} \label{lemma42}
Let $X\in \A$ and $\omega\in E(\Ao)$. Then the following statements are equivalent.
\begin{enumerate}[(i)]
\item $\overline{\omega}$ is an eigenstate of $X$ with eigenvalue $\alpha$.
\item $[ \pi_{\overline{\omega}}(X)\lambda_\omega(I)]=\alpha \lambda_\omega(I)$.
\end{enumerate}
\end{lemma}
\begin{proof}
(i)$\Rightarrow$(ii)
For any $A\in \Ao$ we have
\begin{eqnarray}
\overline{\omega}(AX)=\alpha \omega(A)  
&=& <\pi_{\overline{\omega}}(A)\lambda_\omega(X),\lambda_\omega(I)>   \nonumber \\
&=& <\lambda_\omega(X),\pi(A^\dagger)\lambda_\omega(I)>\nonumber  \\
&=& <\lambda_\omega(X),\lambda_\omega(A^\dagger)>  \nonumber\\
&=& <\pi_{\overline{\omega}}(X)\lambda_\omega(I),\lambda_\omega(A^\dagger)>  ,\nonumber
\end{eqnarray}
and
\begin{eqnarray}
\overline{\omega}(AX)
&=& \alpha \omega(A) \nonumber \\
&=& (\alpha \lambda_\omega(I)|\pi_\omega(A^\dagger)\lambda_\omega(I))\nonumber \\
&=& (\alpha\lambda_\omega(I)|\lambda_\omega(A^\dagger)),\nonumber
\end{eqnarray}
which implies that $\pi_{\overline{\omega}}(X)\lambda_\omega(I)$ is a bounded vector{\color{blue}\footnote{This is the reason why we are using $[ \pi_{\overline{\omega}}(X)\lambda_\omega(I)]$ rather than $ \pi_{\overline{\omega}}(X)\lambda_\omega(I)$.}} in $\Hil_\omega$ and $[\pi_{\overline{\omega}}(X)\lambda_\omega(I)]=\alpha \lambda_\omega(I)$.\\
(ii)$\Rightarrow$(i) This is trivial. 
This completes the proof.
\end{proof}

We will say that $X\in \A$ has a left- (resp., right-)  inverse in $\Ao$, if there exists $B\in \Ao$ such that $BX=I$ (resp., $XB=I$).

Next we define the spectra of an element of $\A$ as follows:
\begin{defn}\label{def43}
Let $X\in \A$. We put
\begin{eqnarray}
\sigma_{\Ao}^L(X)
&:=& \{ \alpha \in\mathbb{C}; \; (X-\alpha I)\; \mbox{ has no left inverse in $\Ao$}\} ,\nonumber \\
\sigma_{\Ao}^R(X)
&:=& \{ \alpha\in\mathbb{C}; \; (X-\alpha I)\;\mbox{ has no right inverse in $\Ao$}\} ,\nonumber \\
\sigma_{\Ao}(X)
&:=& \sigma_{\Ao}^L(X) \cup \sigma_{\Ao}^R(X).\nonumber 
\end{eqnarray}
%{\color{blue}ma non dovrebbe essere $\cap$ sopra?}
{The set $\sigma_{\Ao}(X)$ (resp. $\sigma_{\Ao}^L(X)$, $\sigma_{\Ao}^R(X)$) is called the (resp. {\em left, right}) {\em spectrum} of $X$.} 
\end{defn}
%The following statements are trivial.
It is clear that the maps
\begin{eqnarray}
\alpha \in \sigma_{\Ao}^L(X) \mapsto \overline{\alpha} \in \sigma_{\Ao}^R(X^\ast) \;\;{\rm and}\;\;
\alpha \in \sigma_{\Ao}(X) \mapsto \overline{\alpha} \in \sigma_{\Ao}(X^\ast)\nonumber
\end{eqnarray}
are bijections.\\
\par
\begin{xrem}
For $X\in \A$ we can not define $\sigma_\A^L(X)$, $\sigma_\A^R(X)$ and $\sigma_\A(X)$ because $Y(X-\alpha I)$, $(X-\alpha I)Y$ are not defined for generic $Y\in \A$.
\end{xrem}
\begin{lemma} \label{lemma44}
Let $X\in \A$. Then we have the following
\begin{eqnarray}
Eig(X) \subset \sigma_{\Ao}^L(X) \subset \sigma_{\Ao}(X). \nonumber
\end{eqnarray}
\end{lemma}
\begin{proof}
Take an arbitrary $\alpha \in Eig(X)$. Then there exists a $\|\cdot\|$-continuous state $\omega$ of $\Ao$ satisfying
\begin{equation}
\overline{\omega}(AX)=\alpha \omega(A) \;\;{\rm for \; all} \; A\in \Ao. \label{4.1}
\end{equation}
Now we assume $\alpha \not \in \sigma_{\Ao}^L(X)$. Then there exists a $B\in \Ao$ such that
\begin{eqnarray}
B(X-\alpha I)
= I. \nonumber 
\end{eqnarray}
By \eqref{4.1} we have $\alpha \omega(B)=\overline{\omega}(BX)=\alpha \omega(B)+1$, so $0=1$. This is a contradiction.
Thus $\alpha \in \sigma_{\Ao}^L(X)$. {The inclusion $\sigma_{\Ao}^L(X) \subset \sigma_{\Ao}(X)$ is obvious.}
\end{proof}
\subsection{*-Semisimple CQ*-algebras: a hilbertian approach}
{In the previous sections we considered positive linear functionals on a CQ*-algebra as continuous linear functionals $\omega$ that are positive in $\Ao$. The continuity allows to extend such a functional to the whole space and perform a GNS-like construction.
There are however possible alternative procedures that can be exploited, all closely linked to a variant of the GNS construction which is the main tool for this analysis. One of them is the notion of representable linear functional \cite[Definition 2.4.6]{10} or the notion of invariant positive sesquilinear (ips) form.
\begin{defn}
	Let $(\A,\Ao)$ be a quasi *-algebra. A linear functional $\omega$ on $\A$ is called {\em representable} if
	\begin{itemize}
			\item[(L.1)] $\omega(a^*a)\geq 0, \;, \forall a \in \Ao;$
			\item[(L.2)] $\omega(b^*x^*a)= \overline{\omega(a^*xb)}, \; \forall x\in \A, a,b \in \Ao$;
			\item[(L.3)] $\forall x \in \A$, there exists $\gamma_x>0$, such that 
			$$|\omega (x^*a)| \leq \gamma_x \omega(a^*a)^{1/2}, \quad {\forall a\in\Ao}.$$
		%	{\color{red}for all $a\in\Ao$.}
		\end{itemize}
	\end{defn}

Similarly to the previous sections, every representable linear functional defines a GNS-triple $(\pi_\omega, \lambda_\omega, \D_\omega)$ but now $\lambda_\omega$  takes its values in a dense domain $\D_\omega$ of a Hilbert space $\Hil_\omega$ and $\pi_\omega$ maps $\A$ into the partial *-algebra of operators $ {\mc L}^\dagger(\D_\omega, \Hil_\omega)$. If the quasi *-algebra has a unit, then this representation is cyclic and unique, up to unitary transformations, (see \cite{10}).

The relationship between continuity and representability of a linear functional, discussed in \cite{11} and \cite[Section 3.2]{10}, is still an open problem. For this reason we will suppose that $\omega$ is a {\em continuous} representable linear functional. 
Starting from $\omega$, one can construct a sesquilinear form $\vp_\omega$ by
$$ \vp_\omega(X,Y)= (\pi_\omega(X) \lambda_\omega(I)| \pi_\omega(Y) \lambda_\omega(I) ), \quad X,Y \in \A.$$
It turns out that $\varphi_\omega$ is bounded \cite[Proposition 3.2.2]{10}; that is, $\vp_\omega$ is a member of the set $\PA$ that we are going to define.

}

\begin{defn}\label{defn_ssa}
Let us denote by $\PA$ the family of sesquilinear forms $\vp$ on $\A \times\A$ such that
\begin{itemize}
	\item[(i)] $\vp(X,X)\geq 0, \; \forall X \in \A$;
	\item[(ii)] $\vp(XA,B)=\vp(A, X^*B), \; \forall X\in \A, \forall A,B \in \Ao$;
	\item[(iii)] $\exists$ $\gamma>0$ such that $|\vp(X,Y)|\leq \gamma \|X\|\|Y\|, \; \forall X,Y\in \A$. 
\end{itemize}
By $\SSA$ we denote the subset of elements of $\PA$ for which $\gamma \leq 1$.
\end{defn}
\begin{remark}\label{rem_psf}
We recall that if $\vp$ is a positive sesquilinear form, then $\vp$ satisfies
\begin{itemize}
	\item $\vp(X,Y)= \overline{\vp(Y,X)}, \quad \forall X,Y \in \A$;
	\item $|\vp(X,Y)|^2 \leq \vp(X,X) \vp(Y,Y), \quad \forall X,Y \in \A$.
\end{itemize}
\end{remark}
On the other hand, it is easily shown that to every element $\vp \in \PA$, there corresponds a continuous representable linear functional $\omega_\vp$. Then we go through with our analysis using  sesquilinear forms.

%{\color{red}In \cite{bagintri} a notion of eigenvalue for a quasi *-algebra defined through invariant positive sesquilinear (ips) forms has been proposed. Here we consider shortly the case of a *-semisimple CQ*-algebra (which by definition possesses a sufficient number of ips-forms) and try also to remove the assumption that $(\A[\|\cdot\|], \Ao)$ has a unit.
%}

 To begin with, we remind that every $\vp\in \PA$ allows a GNS-construction as in \cite[Theorem 2.4.1]{10}, that is, there exist a Hilbert space $\Hil_\vp$, a dense subspace $\D_\vp,$ a linear map $\lambda_\vp:\Ao \to \D_\vp$ and a *-representation $\pi_\vp$ of $(\A,\Ao)$ such that, for all $X,Y\in \A$ and $A,B\in \Ao$,
  \begin{equation}\label{eq_GNS} \vp(XA,YB)= (\pi_\vp(X)\lambda_\vp(A)|\pi_\vp(X)\lambda_\vp(A)).\end{equation}
% {\color{blue}\ctrem{I think $\vp_\omega$}}
The triplet $(\pi_\vp, \lambda_\vp, \D_\vp)$ is called the GNS construction for $\vp$. 
Since $\vp\in \PA$ is bounded, $\lambda_\vp(\Ao)$ is dense in $\Hil_\vp$; thus, we can suppose that $\D_\vp=\lambda_\vp(\Ao)$.

We notice that if $(\A,\Ao)$ has a unit $I$, then $\pi_\vp(I)=I_\vp$ the identity operator of $\D_\vp$.

\medskip
The CQ*-algebra $(\A[\|\cdot\|], \Ao)$ is called *-{\em semisimple} if for every $X\neq 0$ there exists $\vp \in \SSA$ such that $\vp(X,X)>0$.

For instance,  if $p\geq 2$, both  the CQ*-algebras $(L^p([0,1]), L^\infty([0,1])$ and  $(L^p(\vp), L^\infty(\vp))$ considered in the examples \ref{esempio ellepi} and \ref{esempio semi}, may be regarded as  *-semisimple
CQ*-algebras.

\begin{defn} Let  $(\A[\|\cdot\|], \Ao)$ be a *-semisimple CQ*-algebra.
	We say that $X\in \A$ has a generalized left inverse if there exists $Y\in \A$ such that
	$$\vp(XA,Y^*B)= \vp(A,B), \quad \forall \vp \in \SSA,\, \forall A,B \in \Ao.$$ 
	Analogously, we say that $X\in \A$ has a generalized right inverse if there exists $Y'\in \A$ such that
	$$\vp(Y'A,X^*B)= \vp(A,B), \quad \forall \vp \in \SSA,\, \forall A,B \in \Ao.$$
	An element $Y$ that is at the same time left- and right generalized inverse of $X$, is called, simply, a generalized inverse of $X$.
\end{defn}
\begin{remark}
It is worth stressing that the generalized inverses need not be unique. 
\end{remark}

In a *-semisimple CQ*-algebra, one can define a weak multiplication by saying that an element $Z\in \A$ is the weak product of $X,Y \in \A$, and it is denoted by $Z=X\mult Y$, if 
$$ \vp(XA,  Y^*B)= \vp(ZA,B) , \quad \forall \vp \in \SSA,\, \forall A,B \in \Ao.$$
Then, for instance, if $(\A,\Ao)$ has a unit $I$,  $Y$ is a generalized right inverse of $X$ if $X\mult Y=I$.

{\begin{defn} \label{defn_gen_eigen}Let $(\A[\|\cdot\|], \Ao)$ be a *-semisimple CQ*-algebra.
A complex number $\alpha$ is called a {\em generalized eigenvalue} of $X\in \A$, if there exist a nonzero $\vp \in \PA$ (called a {\em generalized eigenvector} of $X$) and $A\in\Ao$ such that 
\begin{equation}\label{eqn_eigen}
\mbox{$\varphi(A,A)> 0$ and $\varphi(XA-\alpha A,B)=0, \quad \forall B\in\Ao$.}
\end{equation}

\end{defn}

\begin{proposition}\label{prop_411}
	Let $(\A[\|\cdot\|], \Ao)$ be a *-semisimple CQ*-algebra. The following statements are equivalent.
\begin{itemize}
	\item[(i)] The complex number $\alpha$ is a generalized eigenvalue of $X\in \A$. 
	\item[(ii)] There exists a nonzero $\vp \in \PA$ {and $A\in \Ao$ with $\vp(A,A)>0$} such that
	$$ \vp(XA-\alpha A, XA-\alpha A)=0.$$
		\item[(iii)] There exists a nonzero $\vp\in \PA$ such that ${\rm Ker}(\pi_\varphi(X)-\alpha I_\vp)\neq \{0\}$, where $\pi_\vp$ is the GNS representation constructed from $\vp$.
\end{itemize} 
\end{proposition}}
\begin{proof}
(i)$\Rightarrow$(ii): Suppose that $\alpha$  is a generalized eigenvalue of $X\in\A$. Then, there exist a nonzero $\vp \in \PA$ and $A\in\Ao$ such that \eqref{eqn_eigen} holds.
Let now  $\{ B_n\}\subset\Ao$ be a sequence such that $\| XA-\alpha A-B_n\|\to 0$.
%$B_n\xrightarrow[\|\cdot\|]{}(X-\lambda I)A$.
Then we have
$$ \vp(XA-\alpha A, XA-\alpha A)=\lim_{n\rightarrow\infty}\varphi(XA-\alpha A,B_n)=0.$$

\noindent (ii)$\Rightarrow$(iii): Let $(\pi_\vp, \lambda_\vp, \D_\vp)$ be the GNS construction for $\vp$. Then
$$ \vp(XA-\alpha A, XA-\alpha A) =\|(\pi_\vp(X)-\alpha I_\vp))\lambda_\vp(A)\|^2=0.$$
Hence $(\pi_\vp(X)-\alpha I_\vp)\lambda_\vp(A)=0$ and, since $\|\lambda_\vp(A)\|^2=\vp(A,A)>0$, we conclude that  ${\rm Ker}(\pi_\varphi(X)-\alpha I_\vp)\neq \{0\}$.

\noindent (iii)$\Rightarrow$(i):
Assume that $\alpha \in \mathbb{C}$ is an eigenvalue of $\pi_\vp(X)$, for some $\vp\in \PA$; then there exists $A\in \Ao$ such that $\lambda_\vp(A)\neq 0$ such that $(\pi_\vp(X)-\alpha I_\vp)\lambda_\vp(A)=0$. Then, for every $B\in \Ao$,
$$ \varphi(XA-\alpha A,B)= <(\pi_\vp(X)-\alpha I_\vp)\lambda_\vp(A),\lambda_\vp(B)>=0.$$
This completes the proof.
\end{proof}
\smallskip

\begin{proposition} Let   $(\A[\|\cdot\|], \Ao)$ be  a *-semisimple CQ*-algebra with unit $I$.
Suppose that $\alpha\in {\mathbb C}$ is a generalized eigenvalue of $X$. Then, $X-\alpha I$ has no generalized left inverse.
\end{proposition}
\begin{proof} If $\alpha$ is a generalized eigenvalue of $X$, there exist  $\vp$ and $A\in \Ao$, with $\vp(A,A)>0$ such that $\vp((X-\alpha I)A,B)=0$ for every $B\in \Ao$. Let $Y\in \A$, $\displaystyle Y=\lim_{n\to\infty}B_n$, $B_n\in \Ao$. Then
	$$ \vp((X-\alpha I)A, Y^* C)= \lim_{n\to\infty}\vp((X-\alpha I)A, B_n^* C)=0.$$
	Hence, $X-\alpha I$  has no generalized left inverse.\\
\end{proof}

Let $\D$ be a dense domain in Hilbert space and $K\in {\mc L}^\dag(\D)$. We will say that $K$ is {\em formally normal}  if $K^\dag K=K K^\dag$ or, equivalently if $\|K\xi\|=\|K^\dag\xi\|$ for every $\xi \in \D$.

An element $X\in \A$ is called {\em normal} if 
$$ \vp(XA,XA)= \vp(X^*A,X^*A), \quad\forall \vp \in \PA, \, \forall A\in \Ao.$$ It is clear that $X$ is normal if and only if  $X^*$ is  normal.
\begin{proposition}
Let $X\in \A$. The following statements are equivalent.
\begin{itemize}
	\item[(i)] $X$ is normal.
	\item[(ii)] $\pi_\vp(X)$ is a {formally} normal operator on $\D_\vp =\lambda_\vp(\Ao)$, for every $\vp \in \PA$.
\end{itemize}
\end{proposition}
This is an immediate consequence of the equality
$$ \vp(XA,XA)=\|\pi_\vp(X) \lambda_\vp(A)\|^2, \quad \forall A\in \Ao$$
which holds for every $\vp \in \PA$.

% i.e.,
%$$\vp(X-\alpha I,X-\alpha I)= \vp(X^*-\overline{\alpha} I,X^*-\overline{\alpha} I).$$
%Therefore, by Proposition \ref{prop_411},  $\vp$ is an eigenform of $X$ with generalized eigenvalue $\alpha$ if and only if it is an eigenform of $X^*$ with generalized eigenvalue $\overline{\alpha}$.
\begin{remark}
Let $X\in \A$ be normal and $X=U+iV$, $U=U^*$, $V=V^*$ its cartesian decomposition. Then one easily proves the equality
$$ \vp(UA,VA)=\vp(VA,UA), \quad \forall \vp\in \PA, \, A\in Ao,$$
which can be read as a weak commutation of $U$ and $V$. Indeed, this equality implies that if $U\mult V$ is well defined then also $V\mult U$ is well defined and $U\mult V=V\mult U$.
\end{remark}

{\begin{proposition} Let $X\in \A$ be normal. Then, $\vp$ is a generalized eigenvector of $X$ with generalized eigenvalue $\alpha$ if and only if it is generalized eigenvector of $X^*$ with generalized eigenvalue $\overline{\alpha}$. 
\end{proposition}
\begin{proof}
This follows immediately from the definitions since, for every $\vp \in \PA$ and $A\in \Ao$, one has  $$ \vp(XA-\alpha A,XA-\alpha A)= \vp(X^*A-\overline{\alpha}A, X^*A-\overline{\alpha}A),$$
as can be checked by a direct calculation.
\end{proof}
}
{
	From (iii) of Proposition \ref{prop_411} it follows immediately that
	\begin{corollary} 	If $X=X^*$ then every generalized eigenvalue is real.
\end{corollary}}

\begin{remark}
It is natural to expect that, if $X=X^*$, generalized eigenvectors corresponding to different generalized eigenvalues  are orthogonal, in some sense. For this we need some additional assumption.

Let $\vp, \psi \in \PA$ and let $\pi_\vp$, $\pi_\psi$ the corresponding closed GNS representations. Assume that $\pi_\vp$, $\pi_\psi$ are {\em intertwined} by a bounded operator $T: \Hil_\vp \to \Hil_\psi$ such that $T: \lambda_\vp(\Ao) \to \D(\pi_\psi)$, the domain of $\pi_\psi$, and 
$$T\pi_\vp(X)\lambda_\vp (A)=\pi_\psi(X)T\lambda_\vp (A), \quad \forall X\in \A, A\in \Ao.$$
Suppose now that $\vp$ is a generalized eigenvector of $X$ with eigenvalue $\alpha \in {\mb R}$; then there exists $A\in \Ao$ such that $\vp(A,A)>0$ and $\pi_\vp(X)\lambda_\vp(A)= \alpha \lambda_\vp(A).$
It is easily checked that $$\pi_\psi(X)T\lambda_\vp(A)=\alpha T\lambda_\vp(A).$$ Thus if, $T\lambda_\vp(A)\neq 0$, $T\lambda_\vp(A)$ is a generalized eigenvector of $X$, corresponding to $\alpha$.
Suppose that $\lambda_\psi(B)$, $B\in \Ao$, is an eigenvector of $\pi_\psi (X)$ corresponding to the eigenvalue $\beta \neq \alpha$. Then $\lambda_\psi (B)$ and $T\lambda_\vp(A)$ are orthogonal in $\Hil_\psi.$
\end{remark}
\bigskip
\section{Eigenstates and dynamics} 
Let $H$ be a hermitian element of $\Ao$. 
Since $\Ao$ is a C*-algebra, $e^{itH}\in \Ao$ for all $t\in\mathbb{R}$; so we can define
\begin{equation} \label{eqn_group}
\alpha_t^H(X)
:= e^{itH}Xe^{-itH} \nonumber
\end{equation}
for each $X\in \A$ and $t\in\mathbb{R}$, and $\alpha_t^H$ is a $\ast$-automorphism of $\A$ in the following sense:\\
$\alpha_t^H$ is a bijection and linear map of $\A$ onto $\A$ satisfying
\begin{eqnarray}
\alpha_t^H(I)&=&I, \; \alpha_t^H(AX)=\alpha_t^H(A)\alpha_t^H(X), \nonumber \\
\alpha_t^H(XA)&=& \alpha_t^H(X)\alpha_t^H(A), \; \alpha_t^H(X^\ast)=\alpha_t^H(X)^\ast \nonumber
\end{eqnarray}
for all $A\in \Ao$ and $X\in \A$.
Furthermore, we can easily show the following
\begin{lemma}
$\{ \alpha_t^H\}$ is a $\|\cdot\|$-continuous one-parameter group of $\ast$-automorphisms of $\A$, that is,
\begin{eqnarray}
\alpha_0^H(X)=I, \; \alpha_{s+t}^H(X)=\alpha_s^H(\alpha_t^H(X)).\nonumber
\end{eqnarray}
\end{lemma}
$(\A,\{\alpha_t^H\})$ is called a {\em dynamical system}.
{\begin{xrem}
Suppose that $H\in \A$ and $H^\ast=H$. Then we can not define $\alpha_t^H(X)$ as in \eqref{eqn_group}, because
$$
e^{itH}=\sum_{n=0}^\infty \frac{(it)^n}{n!}H^n 
$$
is not well defined.
%{\color{blue}mi sa che c'è un errore qui: $\alpha_t^H(X)$ non coinvolge $e^{itX}$. Quello che non esiste è $e^{itH}$, se $H\in\A$, ed ovviamente $\alpha_t^H(X)$ non esiste per alcun $X$.}
\end{xrem}}
\begin{lemma} \label{lemma46}
Let $H\in \Ao$.
For any $X\in \A$ we have
\begin{eqnarray}
\lim_{t\rightarrow 0}\| e^{itH}X-X\|&= 0& , \quad\lim_{t\rightarrow 0}\| Xe^{itH}-X\|= 0,\nonumber \\
\lim_{t\rightarrow 0}\|\frac{e^{itH}-I}{t}X-iHX\|&=0&, \quad \lim_{t\rightarrow 0}\| X\frac{e^{itH}-I}{t}-iXH\|=0,\nonumber \\
\lim_{t\rightarrow 0}\| \alpha_t^H(X)-X\|&=0&, \quad \lim_{t\rightarrow 0}\|\frac{\alpha_t^H(X)-I}{t}-i[H,X]\|=0.\nonumber
\end{eqnarray}
\end{lemma}
\begin{proof}
For any $X\in \A$ we have
\begin{eqnarray}
\| e^{itH}X-X\|&\leqq& \| e^{itH}-I\|_0\|X\|\xrightarrow[t\rightarrow 0]{} 0,\nonumber \\
\| \frac{e^{itH}-I}{t}X-iHX\|&\leqq& \|\frac{e^{itH}-I}{t}-iH\|_0\|X\| \xrightarrow[t\rightarrow0]{} 0.\nonumber
\end{eqnarray}
The other statements can be proved in similar way. 
This completes the proof.
\end{proof}
Here we put, for $H\in \Ao$ as in the previous Lemma, 
\begin{eqnarray}
\delta_H(X)
=i[H,X]:=i(HX-XH), \;\; X\in \A.\nonumber
\end{eqnarray}
Then $\delta_H$ is a linear mapping from $\A$ to $\A$ satisfying
\begin{eqnarray}
\delta_H(AX)
&=& \delta_H(A)X+A\delta_H(X), \nonumber \\
\delta_H(XA)
&=& \delta_H(X)A+X\delta_H(A), \nonumber \\
\delta_H(X)^\ast
&=&\delta_H(X^\ast)\nonumber
\end{eqnarray}
for all $X\in \A$ and $A\in \Ao$
and it is called a {\em $\ast$-derivation} of $\A$.
\begin{lemma}\label{lemma47}
Let $H^\ast=H\in \Ao$ and $\omega\in E(\Ao)$. Consider the following
\begin{enumerate}[(i)]
\item $\overline{\omega}$ is an eigenstate of $H$ with eigenvalue $\alpha$
\item $\omega$ is an eigenstate of $H$ with eigenvalue $\alpha$
\item $\overline{\omega}$ is an eigenstate of $e^{itH}$ with eigenvalue $e^{it\alpha}$
\item $\omega$ is an eigenstate of $e^{itH}$ with eigenvalue $e^{it\alpha}$.
\end{enumerate}
Then
\begin{equation}
\begin{array}{ccc}
(i) &\Leftrightarrow&(ii) \\
&\Downarrow&\\
(iii) &\Leftrightarrow&(iv).\\
\end{array}
\nonumber
\end{equation}
\end{lemma}
\begin{proof}
(i)$\Rightarrow$(ii) This is trivial. \\
Take an arbitrary $X\in \A$.
There exists a sequence $\{ A_n\}$ in $\Ao$ such that $\lim_{n\rightarrow \infty}\| A_n-X\|=0$.\\
(ii)$\Rightarrow$(i) Since
$A_n H\in \Ao \xrightarrow[\|\cdot\|]{}XH$
and $\overline{\omega}$ is $\|\cdot\|$-continuous,
it follows that
$\omega(A_nH) \rightarrow \overline{\omega}(XH)$
and $\alpha\omega(A_n)\rightarrow \alpha \overline{\omega}(X)$,
so $\overline{\omega}(XH)=\alpha\overline{\omega}(X)$.
Hence (i) holds.\\
(ii)$\Rightarrow$(iv) By \cite{3} Theorem 2.13 $\omega(A_n e^{itH})=e^{it\alpha}\omega(A_n)$ for all $t\in\mathbb{R}$.
By the $\|\cdot \|$-continuity of $\omega$ we have $\overline{\omega}(Xe^{itH})=e^{it\alpha}\overline{\omega}(X)$, so (iv) holds.\\
(iii)$\Leftrightarrow$(iv) We can proof in the same way as (i)$\Leftrightarrow$(ii). This completes the proof.
\end{proof}
\begin{defn}
Let $H^\ast=H\in \Ao$ and $\omega\in E(\Ao)$. The state $\overline{\omega}$ of $\A$ is said to be {\em invariant under $\alpha_t^H$} if 
\begin{eqnarray}
\overline{\omega}(\alpha_t^H(X))=\overline{\omega}(X) \;\;\; {\rm for \; all}\; X\in \A, \; t\in \mathbb{R}.\nonumber
\end{eqnarray}
\end{defn}

\begin{theorem} \label{theorem49}
Let $H^\ast=H\in \Ao$ and $\omega\in E(\Ao)$. Consider the following
\begin{enumerate}[(i)]
\item $\overline{\omega}$ is an eigenstate of $H$ with an eigenvalue in $\sigma_\A(H).$
\item $\omega$ is an eigenstate of $H$ with an eigenvalue in $\sigma_{\Ao}(H).$
\item $\overline{\omega}$ is invariant under $\alpha_t^H.$
\item $\omega$ is invariant under $\alpha_t^H.$
\end{enumerate}
Then
\begin{eqnarray}
\begin{array}{ccc}
&(i)&\\
&\Downarrow&\\
&(ii)&\\
&\Downarrow&\\
(iii)&\Leftrightarrow&(iv)\\
\end{array}
\nonumber
\end{eqnarray}
\end{theorem}
\begin{proof}
(i)$\Rightarrow$(ii) It follows from Lemma \ref{lemma47} and $\sigma_\A(H)\subset \sigma_{\Ao}(H)$.\\
(iii)$\Rightarrow$(iv) This is trivial.\\
(iv)$\Rightarrow$(iii) Take an arbitrary $X\in \A$. There exists a sequence $\{ A_n\}$ in $\Ao$ such that $\lim_{n\rightarrow\infty}\|A_n-X\|=0$. Then we have
\begin{eqnarray}
\alpha_t^H(A_n)=e^{itH}A_ne^{-itH} \in \Ao \;\;{\rm for \; all} \; n\in \mathbb{N} \nonumber
\end{eqnarray}
and
\begin{eqnarray}
\lim_{n\rightarrow\infty}\|\alpha_t^H(A_n)-\alpha_t^H(X)\|
&\leqq& \lim_{n\rightarrow\infty}\|I-{itH}\|_0 \|e^{-itH}\|_0 \|A_n-X\| \nonumber \\
&=& \lim_{n\rightarrow\infty}\|A_n-X\|\nonumber \\
&=& 0.\nonumber
\end{eqnarray}
Since $\overline{\omega}$ is $\|\cdot\|$-continuous, we have
\begin{eqnarray}
\overline{\omega}(\alpha_t^H(X))
&=& \lim_{n\rightarrow\infty}\omega(\alpha_t^H(A_n)) \nonumber \\
&=& \lim_{n\rightarrow\infty}\omega(A_n) \nonumber \\
&=& \overline{\omega}(X). \nonumber
\end{eqnarray}
(ii)$\Rightarrow$(iv) It follows from \cite{3} Proposition 3.1.\\
This completes the proof.
\end{proof}

\subsection{Ground states}
In {Section 5} we considered the case of $H^\ast=H\in \Ao$ 
%{\color{blue}Non è vero: probabilmente ti riferisci a quanto appena fatto in Sezione 5, sopra. Please check and confirm}.
Here we shall consider the case of $H^\ast=H\in \A$ and $\omega\in E(\Ao)$.
\begin{defn} \label{def410}
The state $\overline{\omega}$ of $\A$ is said to be a {\em ground state for $H$} if
\begin{enumerate}[(i)]
\item $\overline{\omega}$ is a eigenstate for $H$ with an eigenvalue $\alpha_\ast$
\item $<\pi_{\overline{\omega}}(H)\lambda_\omega(B),\lambda_\omega(B)> \geqq \alpha_\ast (\lambda_\omega(B)|\lambda_\omega(B))$ for all $B\in \Ao$.
\end{enumerate}
\end{defn}
We define the spectrum of the form $\pi_{\overline{\omega}}(H)$ as follows:
\begin{defn}\label{def411}
We denote by $Spec(\pi_{\overline{\omega}}(H))$ the set of all $\alpha\in\mathbb{C}$ such that $[\pi_{\overline{\omega}}(H)\lambda_\omega(B)]=\alpha \lambda_\omega(B)$ and $\lambda_\omega(B)\neq 0$ for some $B\in \Ao$, that is, $\lambda_\omega(B)\neq 0\in Ker(\pi_{\overline{\omega}}(H)-\alpha I)$ for some $B\in \Ao.$ This set is called the spectrum of $\pi_{\overline{\omega}}(H)$.
\end{defn}
\begin{theorem}\label{prop412}
Suppose that $H^\ast=H\in \A$ and $\omega\in E(\Ao)$. If $\overline{\omega}$ is a ground state for $H$, then the following statements hold.
\begin{enumerate}[(1)]
\item $-i \; \overline{\omega}(A^\ast \delta_H(A))\geqq 0$ for all $A\in \Ao$.
\item $\overline{\omega}(\delta_H(A))=0$ for all $A\in \Ao$.
\item $\alpha_\ast= \min Spec (\pi_{\overline{\omega}}(H))$.
\end{enumerate}
\end{theorem}
\begin{proof}
(1) Take an arbitrary $A\in \Ao$. Then by Definition \ref{def410} (i) and (ii)
\begin{eqnarray}
-i\;\overline{\omega}(A^\ast \delta_H(A))
&=& <\pi_{\overline{\omega}}(HA-AH)\lambda_\omega(I),\lambda_\omega(A)>\nonumber \\
&=& <\pi_{\overline{\omega}}(H)\lambda_\omega(A),\lambda_\omega(A)>-<\pi_{\overline{\omega}}(H)\lambda_\omega(I),\lambda_\omega(A^\ast A)>\nonumber\\
&=& <\pi_{\overline{\omega}}(H)\lambda_\omega(A),\lambda_\omega(A)>-(\alpha_\ast\lambda_\omega(I)|\lambda_\omega(A^\ast A)) \label{4.2}\\
&=& <(\pi_{\overline{\omega}}(H)-\alpha_\ast I)\lambda_\omega(A),\lambda_\omega(A)>\nonumber \\
&\geqq&0.\nonumber
\end{eqnarray}
(2) Take an arbitrary $A\in \Ao$. Then by (1) we have
\begin{eqnarray}
\overline{\omega}(\delta_H(A^\ast A))
&=& i ( \overline{\omega}(\delta_H(A^\ast)A)-\overline{\omega}(A^\ast \delta_H(A)))\nonumber \\
&=& i(\overline{\omega}((A^\ast\delta_H(A))^\ast)-\overline{\omega}(A^\ast\delta_H(A)))\nonumber \\
&=& \overline{-i\overline{\omega}(A^\ast\delta_H(A))} +i\overline{\omega}(A^\ast\delta_H(A))\nonumber \\
&=& -i\overline{\omega}(A^\ast\delta_H(A))+i\overline{\omega}(A^\ast\delta_H(A))\nonumber \\
&=&0 .\nonumber
\end{eqnarray}
Since $A\in \Ao$ can be expressed as a combination of four positive elements  of $\Ao$, from the functional calculus of C*-algebra, see \cite{2} for instance, we have
\begin{eqnarray}
\overline{\omega}(\delta_H(A))=0. \nonumber
\end{eqnarray}
(3) Take an arbitrary $\alpha\in Spec (\pi_{\overline{\omega}}(H))$. Then, there exists an element $B\in \Ao$ such that $\lambda_\omega(B) \neq 0$ and $[\pi_{\overline{\omega}}(H)\lambda_\omega(B)]=\alpha\lambda_\omega(B)$. By Definition \ref{def410} (ii), we have
\begin{eqnarray}
\alpha_\ast(\lambda_\omega(B)|\lambda_\omega(B))
&\leqq& <\pi_{\overline{\omega}}(H)\lambda_\omega(B),\lambda_\omega(B)>\nonumber \\
&=& \alpha(\lambda_\omega(B)|\lambda_\omega(B)),\nonumber
\end{eqnarray}
so $\lambda_\omega(B)\neq 0$, $\alpha_\ast\leqq\alpha$ because of $\lambda_\omega(B)\neq 0$. Furthermore, since 
\begin{eqnarray}
[ \pi_{\overline{\omega}}(H)\lambda_\omega(I)]=\alpha_\ast \lambda_\omega(I), \nonumber
\end{eqnarray}
we have
\begin{eqnarray}
\alpha_\ast\in Spec [\pi_{\overline{\omega}}(H)]. \nonumber
\end{eqnarray}
Thus (3) holds. This completes the proof.
\end{proof}
\begin{defn}\label{def413}
Let $H^\ast=H\in \A$ and $\omega\in E(\Ao)$.
Suppose the state $\overline{\omega}$ of $\A$ is a ground state of $H$. Then $\overline{\omega}$ is said to be {\em nondegenerate} if $Ker (\pi_{\overline{\omega}}(H)-\alpha_\ast I)=\mathbb{C}\lambda_\omega(I)$. And $\overline{\omega}$ is said to be {\em gapped} if $\overline{\omega}(A^\ast HA)\geqq (\alpha_\ast +\bigtriangleup)\omega(A^\ast A)$ for some $\Delta >0$, for all $A\in \Ao$ with $\lambda_\omega(A)\in (Ker (\pi_{\overline{\omega}}(H)-\alpha_\ast I))^\perp$.
\end{defn}
\begin{theorem}\label{prop414}
Let $H^\ast=H\in \A$ and $\omega\in E(\Ao)$. Suppose $\overline{\omega}$ is a nondegenerate ground state of $H$. Then the following statements are equivalent:
\begin{enumerate}[(i)]
\item $\overline{\omega}$ is a gapped ground state of $H$.
\item There exists a $\bigtriangleup>0$ such that $-i\overline{\omega}(A^\ast \delta_H(A))\geqq \bigtriangleup(\omega(A^\ast A)-|\omega(A)|^2)$ for all $A\in \Ao$.
\end{enumerate}
\end{theorem}
\begin{proof}
(i)$\Rightarrow$(ii) Since $\overline{\omega}$ is nondegenerate, we have
\begin{eqnarray}
(Ker (\pi_{\overline{\omega}}(H)-\alpha_\ast I))^\perp =\{ \lambda_\omega(I)\}^\perp. \nonumber
\end{eqnarray}
For any $A\in \Ao$ we put
\begin{eqnarray}
B:= A-(\lambda_\omega(A)|\lambda_\omega(I))I.\nonumber
\end{eqnarray}
Then $B\in \Ao$ and
$\lambda_\omega(B)\in \{ \lambda_\omega(I)\}^\perp$.
Since $\overline{\omega}$ is a ground state of $H$,
we have
\begin{eqnarray}
\overline{\omega}(B^\ast HB) \geqq (\alpha_\ast +\bigtriangleup)\omega(B^\ast B). \label{4.3}
\end{eqnarray}
Thus we have
\begin{eqnarray}
&&\overline{\omega}(B^\ast HB) \nonumber\\
&=& < \pi_{\overline{\omega}}(H)\lambda_\omega(B),\lambda_\omega(B)>\nonumber \\
&=& <\pi_{\overline{\omega}}(H)(\lambda_\omega(A)-(\lambda_\omega(A)|\lambda_\omega(I))\lambda_\omega(I)),\lambda_\omega(A)-(\lambda_\omega(A)|\lambda_\omega(I))\lambda_\omega(I)> \nonumber \\
&=& <\pi_{\overline{\omega}}(H)\lambda_\omega(A),\lambda_\omega(A)>- \overline{(\lambda_\omega(A)|\lambda_\omega(I))}<\pi_{\overline{\omega}}(H)\lambda_\omega(A),\lambda_\omega(I)>\nonumber \\
&& -(\lambda_\omega(A)|\lambda_\omega(I))<\pi_{\overline{\omega}}(H)\lambda_\omega(I),\lambda_\omega(A)>\nonumber \\
&&+(\lambda_\omega(A)|\lambda_\omega(I))\overline{(\lambda_\omega(A)|\lambda_\omega(I))}<\pi_{\overline{\omega}}(H)\lambda_\omega(I),\lambda_\omega(I)> \nonumber \\
&=& <\pi_{\overline{\omega}}(H)\lambda_\omega(A),\lambda_\omega(A)>-\overline{\omega(A)}(\alpha_\ast \omega(A))\nonumber \\
&&-\omega(A)(\alpha_\ast \lambda_\omega(I)|\lambda_\omega(A))+|\omega(A)|^2(\alpha_\ast|\omega(I)|^2)\nonumber \\
&=& <\pi_{\overline{\omega}}(H)\lambda_\omega(A),\lambda_\omega(A)>-\alpha_\ast|\omega(A)|^2,\nonumber
\end{eqnarray}
so by \eqref{4.2}
\begin{eqnarray}
\overline{\omega}(B^\ast HB)=-i\overline{\omega}(A^\ast \delta_H(A))+\alpha_\ast\omega(A^\ast A)-\alpha_\ast |\omega(A)|^2. \label{4.4}
\end{eqnarray}
In the above equations we used the following equalities
\begin{eqnarray}
[\pi_{\overline{\omega}}(H)\lambda_\omega(I)]
&=& \alpha_\ast \lambda_\omega(I)\nonumber \\
<\pi_{\overline{\omega}}(H)\lambda_\omega(A),\lambda_\omega(I)>
&=& \overline{\omega}(HA)=\overline{\overline{\omega}(A^\ast H)}\nonumber \\
&=& \overline{\alpha_\ast \overline{\omega}(A^\ast)}= \alpha_\ast \overline{\omega}(A).\nonumber
\end{eqnarray}
Since
\begin{eqnarray}
(\alpha_\ast+\bigtriangleup)\omega(B^\ast B)
&=& (\alpha_\ast+\bigtriangleup)\omega((A^\ast -\overline{\omega(A)}I)(A-\omega(A)I))\nonumber \\
&=& (\alpha_\ast+\bigtriangleup)(\omega(A^\ast A)-|\omega(A)|^2),\nonumber
\end{eqnarray}
it follows from \eqref{4.4} that
\begin{eqnarray}
-i\overline{\omega}(A^\ast \delta_H(A))+\alpha_\ast\omega(A^\ast A)-\alpha_\ast |\omega(A)|^2
\geqq (\alpha_\ast+\bigtriangleup)(\omega(A^\ast A)-|\omega(A)|^2).\nonumber
\end{eqnarray}
Thus we have
\begin{eqnarray}
-i\overline{\omega}(A^\ast \delta_H(A)) \geqq \bigtriangleup (\omega(A^\ast A)-|\omega(A)|^2).\nonumber
\end{eqnarray}
(ii)$\Rightarrow$(i) Take an arbitrary $A\in \Ao$ such that $\lambda_\omega(A)\in \{\lambda_\omega(I)\}^\perp$. Then by \eqref{4.2} and assumption (ii) we have
\begin{eqnarray}
\overline{\omega}(A^\ast HA)
&=& <\pi_{\overline{\omega}}(H)\lambda_\omega(A),\lambda_\omega(A)>\nonumber \\
&=& -i \overline{\omega}(A^\ast \delta_H(A))+\alpha_\ast\omega(A^\ast A)\nonumber \\
&\geqq& \bigtriangleup(\omega(A^\ast A)-|\omega(A)|^2)+\alpha_\ast \omega(A^\ast A)\nonumber \\
&=& (\alpha_\ast+\bigtriangleup)\omega(A^\ast A)-\bigtriangleup |\omega(A)|^2\nonumber\\
&=& (\alpha_\ast+\bigtriangleup)\omega(A^\ast A) \nonumber.
\end{eqnarray}
 {In the above equations we used the fact that} $\lambda_\omega(A)\in\{\lambda_\omega(I)\}^\perp$, so $\omega(A)=$$(\lambda_\omega(A)|\lambda_\omega(I))=0$. This completes the proof.
\end{proof}

\section{{A brief digression:} A locally convex $\ast$-algebra constructed from $\Ao(\|\cdot\|)$}

Let $\A(\|\cdot\|)$ be a CQ*-algebra over the C*-algebra $\Ao$.
As shown in Section 3, any $\|\cdot\|$-continuous positive linear functional $\omega$ on $\Ao$ is extendable to a $\|\cdot\|$-continuous (positive) linear functional $\overline{\omega}$ on $\A$ for which the GNS-construction is possible, but the usual operator GNS-construction for $\overline{\omega}$ is impossible in general.
For this reason, in this section we define a locally convex $\ast$-algebra $\A_1$ containing $\Ao$ such that any $\|\cdot\|$-continuous positive linear functional $\omega$ on $\Ao(\|\cdot\|)$ is extendable to an admissible positive linear functional $\overline{\omega}$ on $\A_1$, that is, $\pi_{\overline{\omega}}(X)$ is a bounded linear operator on $\Hil_{\overline{\omega}}$ for all $X\in \A_1$.
For any $N\in\mathbb{N}$ we define a metric space $(\Ao(N),d_N)$ by
\begin{eqnarray}
\Ao(N)
&=& \{ A\in \Ao; \;\;\; \|A\|\leqq N\}, \nonumber \\
d_N(A,B)
&=& \| A-B\|, \;\;\; A,B\in \Ao(N).\nonumber
\end{eqnarray}
Then,
\begin{eqnarray}
\Ao(N_1)
&\subset& \Ao(N_2) \;\;\; {\rm if} \;\;\; N_1\leqq N_2, \nonumber \\
\Ao
&=& \cup_{N\in\mathbb{N}}\Ao(N).\nonumber
\end{eqnarray}
Hence we implement an inductive limit topology $\tau_{ind}$ on $\Ao$ defined by the sequence $\{ (\Ao(N),d_N)\}$ of metric spaces, that is, $\tau_{ind}$-$\lim A_n=A$ if and only if $\{ A_n \} \subset \Ao(N)$ for some $N\in\mathbb{N}$ and $\lim_{n\rightarrow\infty}\| A_n-A\|=0$.

We denote by $\A_1$ the completion of $\Ao$ under the inductive limit topology $\tau_{ind}$.
Then we have the following
\begin{proposition}\label{prop51}
$\A_1$ is a locally convex $\ast$-algebra under the norm $\|\cdot\|$ satisfying
\begin{eqnarray}
\Ao=\cup_{N\in\mathbb{N}}\Ao(N) \subset \A_1=\cup_{N\in\mathbb{N}}\overline{\Ao(N)}[d_N] \subset \A , \label{5.1}
\end{eqnarray}
where $\overline{\Ao(N)}[d_N]$ is the completion of the metric space $\Ao(N)[d_N]$.
\end{proposition}
\begin{proof}
Clearly, $\A_1[\|\cdot\|]$ is a subspace of the Banach space $\A$.
We can define a multiplication of $XY$ of $X$ and $Y$ in $\A_1$.
Indeed, take arbitrary $X,Y\in \A_1$.
There exist sequences $\{ A_n\}$ and $\{ B_n\}$ in $\Ao(N)$ for some $N\in\mathbb{N}$ such that $\lim_{n\rightarrow\infty}\| A_n-X\|=\lim_{n\rightarrow\infty}\|B_n-Y\|=0$.
Then since
\begin{eqnarray}
\| A_mB_m-A_nB_n\|
&\leqq& \|(A_m-A_n)B_m\| +\| A_n(B_m-B_n)\| \nonumber \\
&\leqq& \|B_m\|_0 \|A_m-A_n\|+\|A_n\|_0 \|B_m-B_n\| \nonumber \\
&\leqq& N(\|A_m-A_n\|+\|B_m-B_n\|) , \nonumber
\end{eqnarray}
$\{A_nB_n\}$ is a Cauchy sequence in $\Ao(N)[d_N]$, so $\lim_{n\rightarrow\infty}A_nB_n$ exists in $\overline{\Ao(N)}[d_N]$.
Furthermore, for any sequences $\{ A_n^\prime\}$ and $\{B_n^\prime\}$ in $\Ao(N^\prime)$ such that $\lim_{n\rightarrow\infty}\|A_n^\prime-X\|=\lim_{n\rightarrow\infty}\|B_n^\prime-Y\|=0$.
Then
\begin{eqnarray}
\|A_nB_n-A_n^\prime B_n^\prime\|
&\leqq& \|(A_n-A_n^\prime)B_n\|+\| A_n^\prime(B_n-B_n^\prime)\|\nonumber \\
&\leqq& N^\prime (\|A_n-A_n^\prime\| +\|B_n-B_n^\prime\|) \nonumber
\end{eqnarray}
for all $n\in\mathbb{N}$, so $\|\cdot\|$-$\lim_{n\rightarrow\infty} A_nB_n$ exists in $\A_1$ and it is independent for the method of taking sequences $\{ A_n\}$ and $\{ B_n\}$.
Thus we can define the multiplication $XY$ in $\A_1$ by
\begin{eqnarray}
XY= \| \cdot\|-\lim_{n\rightarrow\infty}A_nB_n, \nonumber
\end{eqnarray}
and it satisfies the following
\begin{eqnarray}
\| XY\|
&=& \lim_{n\rightarrow\infty}\| A_nB_n\| \nonumber \\
&\leqq& \lim_{n\rightarrow\infty}\| A_n\|\|B_n\|_0 \nonumber \\
&\leqq& N \lim_{n\rightarrow\infty} \|A_n\| \nonumber \\
&=& N\| X\|, \nonumber
\end{eqnarray}
and similarly
\begin{eqnarray}
\| XY\| \leqq N\|Y\|, \nonumber
\end{eqnarray}
so the multiplication of $\A_1[|\cdot\|]$ is separating continuous.
Furthermore, since $\| X^\ast\|=\| X\|$, $\A_1[\|\cdot\|]$ is a locally convex $\ast$-algebra.
\eqref{5.1} is trivial. This completes the proof.
\end{proof}

Let $\omega$ be a $\|\cdot\|$-continuous positive linear functional on $\Ao$.
The restriction of the positive linear functional $\overline{\omega}$ on $\A$ to $\A_1$ (we use the same notation $\overline{\omega}$) is a positive linear functional on the locally convex $\ast$-algebra $\A_1[\|\cdot\|]$, so its GNS-construction $(\pi_{\overline{\omega}},\lambda_{\overline{\omega}},\Hil_{\overline{\omega}})$ is possible.
We have the following
\begin{proposition}\label{prop52}
Let $(\pi_{\overline{\omega}},\lambda_{\overline{\omega}},\Hil_{\overline{\omega}})$ be the GNS-construction for a $\|\cdot\|$-continuous positive linear functional $\omega$ on $\Ao$.
Then the $\|\cdot\|$-continuous positive linear functional $\overline{\omega}$ on $\A_1$ is admissible and its GNS-construction $(\pi_{\overline{\omega}},\lambda_{\overline{\omega}},\Hil_{\overline{\omega}})$ satisfies the following properties:
\begin{enumerate}
\item $\Hil_{\overline{\omega}}=\Hil_\omega$.
\item $\pi_{\overline{\omega}}(A)=\pi_\omega(A)$ and $\lambda_{\overline{\omega}}(A)=\lambda_\omega(A)$ for all $A\in \Ao$.
\item For any $X\in \A_1$ there exists an sequence $\{ A_n\} \subset \Ao(N)$ for some $N\in\mathbb{N}$ such that $\pi_\omega(A_n)\mapsto \pi_{\overline{\omega}}(X)$, strongly, namely for any $x\in \Hil_\omega$ $\lim_{n\rightarrow\infty}\pi_\omega(A_n)x=\pi_{\overline{\omega}}(X)x$. So, $\pi_{\overline{\omega}}(\A_1)$ is contained in the bicommutant $\pi_\omega(\Ao)^{\prime\prime}$ of the bounded $\ast$-algebra $\pi_\omega(\Ao)$ on $\Hil_\omega$.
\item For any $X\in \A_1$ there exists an sequence $\{ A_n\}\subset \Ao(N)$ for some $N\in\mathbb{N}$ such that $\lim_{n\rightarrow\infty}\lambda_\omega(A_n)=\lambda_{\overline{\omega}}(X)$.
\end{enumerate}
\end{proposition}
\begin{proof}
Take an arbitrary $X\in \A_1$. Then there exists a sequence $\{ A_n\}\subset \Ao(N)$ for some $N\in\mathbb{N}$ such that $\lim_{n\rightarrow\infty}\| A_n-X\|=0$. Then, for any $n\in\mathbb{N}$ we have 
\begin{eqnarray}
\| (A_n-X)^\ast (A_n-X)\|
&\leqq& \| A_n^\ast(A_n-X)\|+\|X^\ast(A_n-X)\| \nonumber \\
&\leqq& 2N \|A_n-X\|,\nonumber
\end{eqnarray}
so $\|\cdot\|$-$\lim_{n\rightarrow\infty}(A_n-X)^\ast(A_n-X)=0$.
Since $\overline{\omega}$ is $\|\cdot\|$-continuous, we have
\begin{eqnarray}
\lim_{n\rightarrow\infty}\|\lambda_\omega(A_n)-\lambda_{\overline{\omega}}(X)\|^2
= \lim_{n\rightarrow\infty}\overline{\omega} ((A_n-X)^\ast(A_n-X))=0 ,\nonumber
\end{eqnarray}
which implies (1) and (4).
For any $B\in \Ao,
$ $B^\ast A^\ast_n  A_nB \leqq \| A_n\|^2_0 B^\ast B \leqq N^2 B^\ast B$, so for any $X\in \A_1$
\begin{eqnarray}
\| \pi_{\overline{\omega}}(X)\lambda_\omega(B)\|^2
&=&\overline{\omega}(B^\ast X^\ast XB) \nonumber \\
&=& \lim_{n\rightarrow\infty} \omega (B^\ast A_n^\ast A_nB) \nonumber \\
&\leqq& N^2 \omega(B^\ast B) \nonumber \\
&=& N^2 \| \lambda_\omega(B)\|^2, \nonumber
\end{eqnarray}
which implies by (1) and (4) that
\begin{eqnarray}
\pi_{\overline{\omega}}(X)\in B(\Hil_{\overline{\omega}}) \;\;{\rm and}\;\; \|\pi_{\overline{\omega}}(X)\| \leqq N. \label{5.2}
\end{eqnarray}
The statement (2) is trivial.
We show (3).
For any $B\in \Ao$ we have
\begin{eqnarray}
\|\pi_\omega (A_n)\lambda_\omega(B) -\pi_{\overline{\omega}}(X)\lambda_\omega(B)\|^2
&=& \overline{\omega} (B^\ast (A_n^\ast-X)^\ast (A_n-X)B) \nonumber \\
&\leqq& \gamma \|B^\ast (A_n-X)^\ast(A_n-X)B\| \nonumber \\
&\leqq& \gamma \|B\|_0^2\|(A_n-X)^\ast(A_n-X) \| \nonumber \\
&\rightarrow& 0 \;\;\;{\rm as}\;\;\; n\rightarrow\infty .\label{5.3}
\end{eqnarray}
Take an arbitrary $x\in\Hil_\omega$ and any $\varepsilon>0$. 
There exists a $B\in \Ao$ such that $\|\lambda_\omega(B)-x\|<\varepsilon$.
Then, for any $n\in\mathbb{N}$, it follows from \eqref{5.2} that
\begin{eqnarray}
\| \pi_\omega(A_n)x-\pi_{\overline{\omega}}(X)x\|
&\leqq& \| (\pi_\omega(A_n)-\pi_\omega(X))(x-\lambda_\omega(B))\| \nonumber \\
&& + \|(\pi_\omega(A_n)-\pi_\omega(X))\lambda_\omega(B)\| \nonumber \\
&\leqq& N\| x-\lambda_\omega(B)\|+\| (\pi_\omega(A_n)-\pi_\omega(X))(B)\|\lambda, \nonumber
\end{eqnarray}
which implies (3). This completes the proof.
\end{proof}
For the admissible positive linear functionals $\overline{\omega}$ on $\A_1$ we can define the notions of eigenstates, dynamics and  ground states, and obtain the same results studied in Section 4.

 {\section{Conclusions}\label{sect7}

We have proposed a possible extension of the notion of eigenstates for CQ*-algebras, and we have deduced several of their properties. In particular, we have exploited some connections between these states and dynamical systems.

Our analysis of generalized eigenstates in an algebraic settings is far from being completed. We believe that there exist still many aspects which deserve further analysis, both from a mathematical point of view and for their physical applications. Just to cite two interesting topics we plan to consider in a close future, we mention the case of non Hermitian Hamiltonian $H$ in the definition of the dynamics, which has triggered the interest of many scholars in the past decades, \cite{benbook}, and the construction of generalized eigenstates and eigenvalues in the context of Section 6 . 

\section*{Acknowledgements}

F. B., C.T. and S. T. acknowledge partial financial support from Palermo University, and from G.N.F.M.  and G.N.A.M.P.A. of the INdAM. H. I. acknowledges partial financial support from Kyushu Sangyo University.
 S. T. acknowledge the project: D26 PREMIO GRUPPI RIC2023 TRIOLO SALVATORE. F.B. and S.T. acknowledge the project: ICON-Q, Partenariato Esteso NQSTI - PE00000023, Spoke 2. F. B. also acknowledge partial support from the PRIN grant {\em Transport phenomena in low dimensional
 	structures: models, simulations and theoretical aspects} - project code 2022TMW2PY - CUP B53D23009500006. 

\section*{Ethics statement}

This work did not involve any active collection of human data.

\section*{Data accessibility statement}

This work does not have any experimental data.

\section*{Competing interests statement}

We have no competing interests.}







% \end{thebibliography}

\end{document}